\documentclass[11pt]{article}
\usepackage{graphicx}%
\usepackage{multirow}%
\usepackage{amsmath,amssymb,amsfonts,a4wide}%
\usepackage{amsthm}%
\usepackage{mathrsfs}%
\usepackage{dsfont}%
\usepackage[title]{appendix}%
\usepackage{xcolor}%
\usepackage{textcomp}%
\usepackage{manyfoot}%
\usepackage{booktabs}%
\usepackage{algorithm}%
\usepackage{algorithmicx}%
\usepackage{algpseudocode}%
\usepackage{listings}
\usepackage{tikz}%

\theoremstyle{plain}%
\newtheorem{theorem}{Theorem}[section]
\newtheorem{proposition}[theorem]{Proposition}%
\newtheorem{lemma}[theorem]{Lemma}%
\newtheorem{corollary}[theorem]{Corollary}%

\theoremstyle{definition}%
\newtheorem{example}[theorem]{Example}%
\newtheorem{remark}[theorem]{Remark}%
\newtheorem{definition}[theorem]{Definition}%

\numberwithin{equation}{section}

\newcommand{\cB}{{\cal B}}

\newcommand{\cF}{{\cal F}}

\newcommand{\cK}{{\cal K}}
\newcommand{\cL}{{\cal L}}

\newcommand{\cO}{{\cal O}}
\newcommand{\cP}{{\cal P}}

\newcommand{\cX}{{\cal X}}
\newcommand{\cY}{{\cal Y}}
\newcommand{\cZ}{{\cal Z}}

\newcommand{\Id}[1]{\mathds{1}_{#1}}
\newcommand{\GR}{\mathds{R}}

\newcommand{\GN}{\mathds{N}}
\newcommand{\GH}{\mathds{H}}
\newcommand{\GD}{\mathds{D}}

\newcommand{\GI}{\mathds{I}}

\newcommand{\bP}{\mathbb P}
\newcommand{\bE}{\mathbb E}

\newcommand{\zB}{\mathscr{B}}
\newcommand{\zA}{\mathscr{A}}
\newcommand{\zK}{\mathscr{K}}
\newcommand{\sr}{\,|\,}

\newcommand{\conver}{\mathop{\longrightarrow}}

\newcommand{\conv}{\conver}

\newcommand{\intau}{\ \conver_{\tau  }\ }

\newcommand{\bequ}[1]{\begin{equation}\label{#1}}
	\newcommand{\eequ}{\end{equation}}

\newcommand{\starindistdisttop}[1]{\ {\stackrel{*}{\overrightarrow{\longrightarrow}}_{\mathcal{D}(#1)}}\ }
\newcommand{\indistdisttop}[1]{\ {\stackrel{\phantom{*}}{\overrightarrow{\longrightarrow}}_{\mathcal{D}(#1)}}\ }

\begin{document}
	
\title{The Riesz representation in submetric spaces\\ and its probabilistic aspects}


\author{Adam Jakubowski\footnote{E-mail:adjakubo@mat.umk.pl}\\[3mm]
Faculty of Mathematics and Computer Science\\
Nicolaus Copernicus University, Toruń, Poland}
\date{}

\maketitle

\begin{abstract} A {\em submetric space} is a topological space equipped with a continuous metric that generates a metric topology (usually) weaker than the original one (e.g.,  a~separable Hilbert space with the weak topology). We show that a positive linear functional on the space of bounded continuous functions on a submetric space is given as an integral with respect to a tight measure if, and only if, it is continuous in the strict topology.  Neither local compactness nor complete regularity is required.

We link the equicontinuity of a family of functionals in the strict topology to the uniform tightness of the underlying measures, and we provide a probabilistic representation (Skorokhod's parametrization) for such families. Finally, we identify the finest topology on the space of probability measures on a submetric space in which uniform tightness implies relative compactness. In non-metric spaces this topology is strictly stronger than the topology of weak convergence of probability measures.
\end{abstract}


\noindent{\em MSClassification 2020:} 28C05, 60B05, 60B10, 54A20, 54E99.

\tikzset{every picture/.style={line width=0.75pt}} 

\section{Introduction}\label{sec1}

The Riesz representation \cite{Riesz09} is one of the cornerstones of the functional analysis. It asserts that (in contemporary terminology and notation) any linear and bounded functional $F$ on the space $C([0,1])$ of continuous functions on the unit interval $[0,1]$ is, in fact, an integral with respect to a signed measure on Borel subsets of $[0,1]$. This form of the Riesz representation is usually presented in an equivalent though apparently weaker statement: if $F\,:\, C([0,1]) \to \GR^1$ is linear, {\em positive} (i.e. $F(f) \geq 0$ whenever $f \geq 0$) and {\em normalized} (i.e. $F(1) = F(\Id{[0,1]}) = 1$) then 
\begin{equation}\label{eq:1}
 F(f) = \int_{[0,1]} f(x)\,\mu(dx), \quad f\in C([0,1]),
\end{equation}
for some (unique) {\em probability} measure $\mu$ on $\zB_{[0,1]}$. 

There is a huge literature devoted to generalizations and consequences of the Riesz representation (e.g., Riesz spaces, Banach lattices, pre-Riesz spaces, etc.). Here we shall focus on the simplest thread, related to extending  (\ref{eq:1}).

Within the first half of the XX century, the Riesz representation had been developed in three main directions.

First, Radon in his seminal paper \cite{Rad13} considered functionals of continuous functions with compact support in Euclidean spaces. This heading was completed by  
 Halmos in \cite[\S 56]{Halm50} where he gave a theory for locally compact spaces. See, e.g., \cite{Mez07} for contemporary variations on this topic. 
 
 In parallel, first Markoff \cite{Mar38} and then Alexandroff  \cite{Alex40}, discussed an abstract representation  \begin{equation}\label{eq:2}
 	F(f) = \int_{\cX} f(x)\,\mu(dx), \quad f\in CB(\cX,\tau),
 \end{equation}
 where $(\cX,\tau)$ is a topological space, $CB(\cX,\tau)$ is the space of $\tau$-continuous and {\em bounded} functions on $(\cX,\tau)$ and $\mu$ is a set function allowing integration of functions $f\in CB(\cX,\tau)$. Markoff and Alexandroff obtained representation (\ref{eq:2}) for normal spaces $(\cX,\tau)$ and finitely additive set functions $\mu$.
 
 Finally, Saks \cite{Saks38}, developing some formalism by Banach \cite{Ban37}, obtained (\ref{eq:2}) for metric compacts $(\cX,\tau)$ and Borel measures $\mu$. The crucial step was done by Kakutani \cite{Kak41}, who derived the representation (\ref{eq:2}) for general compacts $(\cX,\tau)$.

It is important for the present paper to stress the inspiring role of John von Neumann in the development of this theory. Markoff \cite[p. 166]{Mar38} acknowledges the influence of \cite{vNeu35}. Kakutani \cite[p. 1009]{Kak41} directly refers to ``Princeton lectures" by von Neumann \cite{vNeu41}. And Halmos made notes from these lectures and the idea of defining measures ``from compacts to Borel sets" can be easily traced  in his important book. 

  With Kakutani's result and the progress of general topology in the thirties of XX century, it was natural to extend the theory of (\ref{eq:2}) to {\em Tychonoff spaces} (completely regular plus $T_1$ spaces). Roughly speaking, any Tychonoff space $(\cX,\tau)$ can be identified with a subset  of  some Tychonoff cube and bounded continuous functions on $(\cX,\tau)$ can be extended to the Stone-\v{C}ech compactification $\beta \cX$ of $\cX$. Then Kakutani's result works; the problem is that the obtained measure is defined on Borel subsets of $\beta \cX$ and after the restriction to $\cX$ it is only finitely additive. 

It was Buck \cite{Buck52}, \cite{Buck58} who realized (in the framework of locally compact spaces) that the measure $\mu$ in the integral on the right-hand side of (\ref{eq:2}) is {\em $\tau$-tight} (i.e. $\sup \{ \mu(K)\,;\, K\in \cK_{\tau} \} = 1$, where $\cK_{\tau}$ is the family of $\tau$-compact sets) if, and only if, the integral is continuous in {\em the strict topology}.  The latter means that there exists a function $v \geq 0$ which is ``vanishing at infinity" (i.e., bounded and satisfying  $\{x\sr v(x) \geq \varepsilon\} \in \cK_{\tau}$, for every $\varepsilon > 0$) and such that
\begin{equation}\label{eq:3}
	\big| \int_{\cX} f(x)\,\mu(dx)\big| \leq \sup_{x\in\cX} |f(x)| v(x) = \|f v\|_{\infty}, \quad f\in CB(\cX,\tau).
\end{equation}
In what follows we will denote the family of functions vanishing at infinity on $(\cX,{\tau}) $ by $V_{\infty}(\tau)$. Buck's idea was rediscovered by probabilists in complete and separable metric spaces (\cite[Theorem 1, p.315]{GiSk74}, see also \cite[Appendix A]{LaSz06}). Summarizing, our problem (\ref{eq:2}) admits a $\tau$-tight solution when $(\cX,\tau)$ is a Tychonoff space and the functional $F$ is continuous in the strict $\tau$-topology (see \cite[Theorem 3, p. 141]{Jar81}). Relatively recent developments of the theory within this framework can be found, e.g., in \cite{Now16}.

At this point it might be instructive to recall two statements from the important paper \cite{Sent72} by Sentilles.  On p. 311 he wrote: ``(...) extend this kind of representation to its last reasonable setting, completely regular Hausdorff spaces. This setting appears to be ultimate in the sense of Hewitt's example of a regular space upon which the only continuous functions are constant". And on p. 313 we can read ``the restriction of tightness is too strong, for the compacta in $\cX$ generate neither $\text{Ba}(\cX)$ nor $\text{B}(\cX)$" ($\text{Ba}(\cX)$ are ``Baire sets" and $\text{B}(\cX)$  are ``Borel sets''). These sentences apparently look correct, and have not been questioned yet.

We shall show in the present paper that the Riesz representation for strictly continuous functionals holds in a large class of topological spaces, so-called submetric spaces, which need not even be regular, yet still admit rich families of continuous functions.  With the formalism of submetric spaces we are able to avoid employing the full power of the Axiom of Choice and achieve the aim of maximal simplicity. In particular, we remain in the world of tight measures only, which is the most important one for applications. 

A topological space  $\big(\cX,\tau\big)$ is called {\em submetric}, if the topology $\tau$ (i.e., the family of $\tau$-open sets $\cO_{\tau}$) contains  a metric topology $\tau_d$ generated by some metric $d$ on $\cX$. The term was coined in \cite{Grue84}. However, that article is of limited relevance from our perspective, as it focuses on how to eliminate the metric from the definition; by contrast, we demonstrate how much can be gained from the existence of a continuous metric $d$ on $(\cX,\tau)$.
The position of submetric spaces within the hierarchy of classes of topological spaces can be described by the following diagram.
\begin{center}
	\begin{tikzpicture}[x=0.75pt,y=0.75pt,yscale=-1,xscale=1]

	\draw (70,125) node [anchor=north west][inner sep=0.75pt]   [align=left] {MS};
	\draw (98,125) node [anchor=north west][inner sep=0.75pt]    {$\subset $};
	\draw (120,125) node [anchor=north west][inner sep=0.75pt]   [align=left] {SS};
	\draw (140,125) node [anchor=north west][inner sep=0.75pt]    {$\cap $};
	\draw (155,125) node [anchor=north west][inner sep=0.75pt]   [align=left] {TS};
	\draw (180,122) node [anchor=north west][inner sep=0.75pt]  [rotate=30]  {$\subset $};
	\draw (185,130) node [anchor=north west][inner sep=0.75pt]  [rotate=-30]  {$\subset $};
	\draw (200,112) node [anchor=north west][inner sep=0.75pt]   [align=left] {SS};
	\draw (200,138) node [anchor=north west][inner sep=0.75pt]   [align=left] {TS};
	\draw (224,137) node [anchor=north west][inner sep=0.75pt]  [rotate=30]  {$\subset $};
	\draw (228,113) node [anchor=north west][inner sep=0.75pt]  [rotate=-30]  {$\subset $};
	\draw (247,125) node [anchor=north west][inner sep=0.75pt]   [align=left] {HS};
\end{tikzpicture}
\end{center}
In this diagram MS means ``metric spaces'', SS $=$ ``submetric spaces'', TS $=$ ``Tychonoff spaces'', HS $=$ ``Hausdorff spaces''. Examples showing that each of the above inclusions is strict, can be found in \cite[Section 2]{Jak26}. A list of basic properties of submetric spaces (e.g., compactness $=$ sequential compactness!) which make the class of submetric spaces particularly useful in probabilistic considerations, is given in \cite[Section 3]{Jak26}. 

The family of submetric spaces is really large.
In Example \ref{ExSmir} we discuss the case of Smirnov space, which is naturally submetric, but is not regular. In Example 
\ref{ExHilb} we discuss  the weak topology on a separable Hilbert space $(\GH,\tau_w)$. This example is typical by two reasons. First, it is representative for weak topologies on linear topological spaces. Second,  $(\GH,\tau_w)$ is a submetric space, because there exists a countable family of continuous functions which separate points (``CCSP'') in $\GH$ , e.g., the scalar products $\{f_i(x) =\langle x,e_i\rangle\,:\, i\in \GN\}$, where $\{e_i\}$ is an orthonormal basis in $\GH$. Given such a family $\{f_i\}$ one can define
 a {\em continuous} metric on $\GH$ by the formula
 \begin{equation}\label{eq:4}
 	d(x,y) = \sum_{i=1}^{\infty} \frac{1}{2^i} \frac{|f_i(x) - f_i(y)|}{1 + |f_i(x) - f_i(y)|}.
 \end{equation} 
  This recipe for constructing a submetric structure on a topological space is highly flexible (where applicable) and is frequently used in various configurations in solving existence problems for stochastic partial differential equations (SPDEs) and establishing the existence of invariant measures. We shall demonstrate its applicability in Section \ref{SecLaw}, where we show that the space of tight probability measures on a submetric space can be endowed with a submetric topology.
 
 Historically,  this procedure was first applied to its full extent in the author's paper \cite{Jak97NS}, which introduced the so-called $S$ topology on the Skorokhod space  $\GD([0,1])$ of real c\`adl\`ag functions on $[0,1]$.  The $S$ topology is defined via convergent sequences, it is non-metric but submetric, and it admits tractable criteria for relative compactness. As it is sequential but non-metric, it remains unknown whether $(\GD([0,1]), S)$ is completely regular or even regular.
 
 In fact, an application of the $S$ topology to some problems of financial mathematics, given in \cite{CKPS21}, was a motivation for the present paper. In \cite{CKPS21}, the authors needed the Riesz representation on $(\GD([0,1]), S)$. Because of the lack of regularity, they defined a new topology $S^*$ on $\GD([0,1])$ that $1^o$ converted the Skorokhod space into a Tychonoff space and $2^o$ preserved the family of $S$-compact sets. The choice of $S^*$ was quite natural - it was the topology generated by all $S$-continuous (hence sequentially continuous) functions on $\GD([0,1])$. While bringing the solution, this procedure is rather indirect and involves the Tychonoff/Stone/\v{C}ech machinery related to the Axiom of Choice. Furthermore, Example \ref{ExSmir} shows that in some submetric spaces the step $2^o$ does not hold, so in general the method of \cite{CKPS21} cannot be directly repeated. Nevertheless, after resolving some subtleties, it is possible to prove our Theorem \ref{main} following the line of \cite{CKPS21}. See Remark \ref{rem:Kiiski} for all details.
 
 We prefer a direct proof of the Riesz representation theorem in submetric spaces, as it highlights the untapped potential of the von Neumann–Halmos measure construction method. Moreover, it suggests that complete regularity (which is necessary for handling closed sets) is irrelevant in certain problems on submetric spaces, as compact sets are sufficient.

The paper is organized as follows. In Section \ref{SecMain} we formulate and prove a Riesz-type representation for submetric spaces (Theorem \ref{main}). We also study the example of the Smirnov space in detail. 

In Section \ref{SecUT}, we connect the property of uniform tightness of a family of probability measures to the equicontinuity in the strict topology of the corresponding integral functionals. The example of a separable Hilbert space with the weak topology is basic for this section. As an extra tool we provide a probabilistic coupling of the uniformly tight family of probability measures — the so-called Skorokhod parametrization (Theorem \ref{th:SkPar}). This result is, in fact, the strong version of the direct Prohorov theorem for submetric spaces. In Remark \ref{SCM} we briefly refer to the Stochastic Compactness Method in SPDEs - the main area of application of the machinery developed in this section.

In Section \ref{SecLaw}, we define a sequential topology on the set of tight probability measures on a submetric space. This topology is itself submetric and coincides with the usual weak convergence of tight probability measures on metric spaces. In general, however, it is strictly stronger than weak convergence, as demonstrated by the classic counterexample due to Fernique \cite{Fern67}. The uniform tightness of a family of probability measures is sufficient for relative compactness in this topology (the direct Prohorov theorem). Crucially, in many settings (e.g., on $(\GH,\tau_w)$), this condition is also necessary (the converse Prokhorov theorem in submetric spaces). 

In Appendix \ref{ApHalm}, we carefully examine the von Neumann–Halmos content method for generating measures and show that, at a basic level, it holds in Hausdorff spaces.

In Appendix \ref{sec:tight}, we prove that within the context of submetric spaces, any tight measure is both regular and Radon; thus, there is no need to distinguish between these notions in our paper (see Proposition \ref{prop: unif}).

Appendix \ref{Ap:RandEl} contains definitions and results related to random elements taking values in a submetric space.

\section{The Riesz representation in submetric spaces}\label{SecMain}

\begin{theorem}\label{main}
Let $(\cX,\tau)$ be a submetric space and let $CB(\cX,\tau)$ be the space of bounded continuous functions on  $(\cX,\tau)$.

Let $\Lambda: CB(\cX,\tau) \to \GR^1$ be a positive linear functional, satisfying
$\Lambda(\Id{\cX}) = 1$.

Then there exists a $\tau$-tight probability measure $\mu$ defined on Borel sets $\zB_{\tau}$ and such that 
\begin{equation}\label{eq:main1}
	 \Lambda(f) = \int_{\cX} f(x)\,\mu(dx),\ f \in CB(\cX,\tau),
\end{equation}
if, and only if, $\Lambda$ is continuous in the strict $\tau$-topology, i.e., there exists a function $v \in V_{\infty}(\tau)$ such that for each $f \in CB(\cX,\tau)$ 
\begin{equation}\label{eq:main2} |\Lambda(f)| \leq \sup_{x\in\cX} |f(x) v(x)|\ (= \|fv\|_{\infty}).
\end{equation}

The measure $\mu$ is unique.
\end{theorem}

  \begin{proof}
Let $\mu$ be a $\tau$-tight probability measure on a submetric space $\big( \cX, \tau)$. We essentially follow \cite{Buck58}, and literally \cite[p. 142]{Jar81}. Choose {\em compact} subsets $K_n \subset \cX$ such that $K_n \subset K_{n+1}$ and 
\begin{equation}
	\mu\big( K_n \big) > 1 - 4^{-n-1},\quad n = 1,2, \ldots,
\end{equation}
and let $K_0 = \emptyset$. Define
\begin{equation}
	v(x) = \begin{cases} 2^{-n} &\text{if $x\in K_n\setminus K_{n-1}$}\\
		0&\text{if $x\not\in \bigcup_{n=1}^{\infty} K_n$}.
	\end{cases}
\end{equation}
Clearly, $v(x) \in V_{\infty}$  Then we have for $f\in CB(\cX,\tau)$
\begin{align}
	\big|\int_{\cX} f(x) \mu(dx)\big| &= \big| \int_{ \bigcup_{n=1}^{\infty} K_n} f(x) \mu(dx)\big| \leq  \int_{ \bigcup_{n=1}^{\infty} K_n} |f(x)|v(x) \frac{1}{v(x)}\mu(dx)  \nonumber \\
	& \leq \sup_{x\in\cX} |f(x)v(x)| \cdot   \int_{\bigcup_{n=1}^{\infty} K_n} \frac{1}{v(x)}\mu(dx) \nonumber\\
	&= \|f\cdot v\|_{\infty} \cdot \sum_{n=1}^{\infty} \int_{K_n\setminus K_{n-1}}\frac{1}{v(x)} \mu(dx)\nonumber\\
	&\leq \|f\cdot v\|_{\infty} \cdot \sum_{n=1}^{\infty} 2^n \mu(K_{n-1}^c) < \|f\cdot v\|_{\infty} \cdot \sum_{n=1}^{\infty} 2^n \frac{1}{4^{n}} = \|f\cdot v\|_{\infty}.\label{nes}
\end{align}
In other words, the linear functional $CB(\cX,\tau) \ni f \mapsto \Lambda_{\mu}(f) =  \int_{\cX} f(x) \mu(dx)$ is continuous in the {\em strict} topology.  

In order to prove the converse, we shall adapt the classic formalism given in \cite[Chpt. X]{Halm50} for locally compact spaces. 

Let $\Lambda : CB(\cX,\tau) \to \GR^1$ be linear, positive, $\Lambda(\mathds{1}_{\cX}) = 1$,  and satisfies 
\[ |\Lambda(f)| \leq \sup_{x\in\cX} |f(x)v(x)|,\]
where $v\in V_{\infty}$. Recall that $0 \leq v(x) \leq \sup_{y\in\cX} v(y) = \|v\|_{\infty} < +\infty$ and $\{x\sr v(x) \geq \varepsilon\} \in \cK_{\tau}$, for each $\varepsilon > 0$.

Let us define a function $\lambda$ on $\tau$-compact subsets $\cK_{\tau}$ of $\cX$.  
\[ \lambda(K) = \inf\{ \Lambda(f)\sr f\in CB(\cX,\tau), f \geq  \Id{K}\},\quad K\in \cK_{\tau}.\] 
We will show that $\lambda$ is a $\tau$-content in the sense of Definition \ref{halmcont}. Obviously $\lambda$ is non-negative, bounded by $1$ (so finite), monotone and $\lambda(\emptyset) = 0$. Take $K_1, K_2 \in \cK_{\tau}$ and let $f_1, f_2 \in CB(\cX,\tau)$ be such that $f_1 \geq \Id{K_1}, f_2 \geq K_2$. Then $f_1 + f_2 \geq \Id{K_1 \cup K_2}$ and so 
$ \lambda(K_1\cup K_2) \leq \Lambda(f_1) + \Lambda(f_2)$. It follows that
\[ \lambda(K_1\cup K_2) \leq \inf \{\Lambda(f)\sr f\geq \Id{K_1} \} +  \inf \{\Lambda(g)\sr g\geq \Id{K_2} \} = \lambda(K_1) + \lambda(K_2).\]
Hence we have the subadditivity of $\lambda$. To proceed further we need a lemma.
\begin{lemma}\label{Separ}
	Let $K_1$ and $K_2$ be {\em disjoint} $\tau$-compact subsets of a submetric space $\big( \cX, \tau)$. Then there exists a continuous function $h : \cX \to [0,1]$ such that $ h(x) = 1$ on $K_1$ and $h(x) = 0$ on $K_2$. \end{lemma}
\begin{proof}[Proof of Lemma~{\upshape\ref{Separ}}] Take any metric $d$ which is $\tau$-continuous and let $\tau_d$ be the metric topology generated by $d$.
	Let $K_1, K_2 \subset \cX$ be $\tau$-compact. Then they are also $\tau_d$-compact. If $K_1$ and $K_2$ are disjoint, then for each $x \in \cX$ we have $d(x, K_1) + d(x,K_2) > 0$ and we may define 	
	\begin{equation}\label{funcha} h(x) = \frac{d(x,K_2)}{d(x, K_1) + d(x,K_2)}.
	\end{equation}

\end{proof}   
\begin{remark}
	If $K_1$ and $K_2$ are disjoint compact sets in a Hausdorff space, then there are {\em disjoint} open sets $U_1 \supset K_1$ and $U_2 \supset K_2$. In submetric spaces the separation of $K_1$ and $K_2$ can be directly obtained by setting $U_1 = \{ x\,:\,h(x) > 1/2\}$ and $U_2 = \{ x\,:\, h(x) < 1/2\}$, where $h$ is the function from Lemma~{\upshape\ref{Separ}}. 
\end{remark}
Now assume that $K_1 \cap K_2 = \emptyset$ and let $h$ be the function given by the above lemma. Choose $f \geq \Id{K_1 \sqcup K_2} = \Id{K_1} +  \Id{K_2}$. Then $f\cdot h \geq \Id{K_1}$ and $f\cdot (1- h) \geq \Id{K_2}$ and therefore 
\[ \lambda(K_1) + \lambda(K_2) \leq \Lambda(f\cdot h) + \Lambda(f(1-h)) = \Lambda(f).\]
Hence we have
\[ \lambda(K_1) + \lambda(K_2) \leq \inf\{ \Lambda(f)\sr f \geq \Id{K_1 \sqcup K_2}\} = \lambda(K_1 \sqcup K_2),\]
and the additivity of $\lambda$ on $\cK_{\tau}$ follows. 

We have proved that $\lambda$ is a $\tau$-content.  Applying Theorem \ref{apmain} we find a {\em finite measure} $\mu$ on the $\sigma$ algebra $\zK_{\tau} = \sigma(\cK_{\tau})$ such that
\[  \lambda(K) \leq \mu(K),\quad K\in \cK_{\tau}.\]
We shall extend the measure $\mu$ to Borel subsets $\zB_{\tau} = \sigma(\cO_{\tau})$. Let $\lambda_*$ and $\mu^*$ be the inner content induced by $\lambda$ (Definition \ref{inncont}) and the outer measure induced by $\lambda$ (Definition \ref{outmeas}), respectively.

Let $K_n = \{ x\sr v(x) \geq (1/n)\}$, $n=1,2,,\ldots$ and let $K_{\infty} = \bigcup_{n=1}^{\infty}$.   Fix $n\in\GN$, take a $\tau$-compact set $L \subset K_n^c$ and let $h_n$ be the continuous function given by Lemma \ref{Separ} for the pair $L$ and $K_n$. 
We have by the definition of $\lambda$ and the strict continuity of $\Lambda$, 
\begin{equation}\label{eq:conc1} \lambda(L) \leq \Lambda(h_n) \leq \sup \{x\,:\, h_n(x) v(x)\} \leq (1/n).
\end{equation}
Keeping in mind that $K_n^c \in \zK_{\tau}\cap \cO_{\tau}$ we obtain (see Theorem \ref{outhalm} (ii)) that
\begin{equation} \label{eq:conc2} (1/n) \geq \lambda_*(K_n^c) = \mu^*(K_n^c) = \mu(K_n^c).
\end{equation}
In particular, $ 0 = \mu(K_{\infty}^c) = \mu^*(K_{\infty}^c)$, and the measure $\mu$ can be naturally extended to a {\em tight} probability measure on (at least) the $\sigma$-algebra of all Borel sets $\zB_{\tau}$.

It remains to identify $\Lambda(f)$ with $\int_{\cX} f(x)\,\mu(dx)$. A part of the proof is fairly standard. 
Take $U \in \cO_{\tau}$ and $f\in CB(\cX,\tau), f \geq \Id{U}$. Then
\[ \mu(U) = \mu^*(U) = \lambda_*(U) \leq \Lambda(f). \]
Repeating step by step the proof of \cite[Theorem B, p. 245]{Halm50} we obtain the relation 
\begin{equation} \label{eq:l-f} \int_{\cX} f(x) \mu(dx) \leq \Lambda(f), \quad f\in CB(\cX,\tau), f \geq 0.
\end{equation} 
Let us consider the functional $\Theta(f) = \Lambda(f) - \int_{\cX} f(x)\mu(dx)$. It is a positive linear functional on $CB(\cX,\tau)$. The $\tau$-content $\theta$ generated by the functional $\Theta$ satisfies
\begin{align*} 0 \leq \theta(K) &= \big\{ \Lambda(f) - \int_{\cX} f(x) \mu(dx)\sr f \in CB(\cX,\tau), f \geq \Id{K}\big\} \\ &\leq  \inf\big\{ \Lambda(f) \sr f \in CB(\cX,\tau), f \geq \Id{K}\big\} - \mu(K) = \lambda(K) - \mu(K) \leq 0.\end{align*}
Therefore we have $\lambda(K) = \mu(K)$, $K\in \cK_{\tau}$.

Now we are ready to prove that $\mu(\cX) =1$. To see that, let us notice that
if $1\geq f \geq \Id{K_n}$, then by the continuity in the strict topology
\[ 1 - \Lambda(f) =|\Lambda(\Id{\cX} - f)| \leq \sup_{x \in K_n} |1 - f(x)| \|v\|_{\infty} + \sup_{x \in K_n^c} (1 - f(x)) v(x) < 1/n. \]
This means that  $ \Lambda(f) > 1 - 1/n$, hence
\[ \mu(K_n) = \lambda(K_n) = \inf\{ \Lambda(f)\sr f \geq \Id{K_n}\} \geq 1 - 1/n.\] 
Thus $1 \geq \mu(\cX) \geq \sup_n \mu(K_n) = 1$ and our claim follows. The rest of the procedure of identification is immediate. Take a general $f\in CB(\cX,\tau)$ and let $C= \sup_x |f(x)|$. Then both $f + C \geq 0$ and $-f + C \geq 0$ and therefore by (\ref{eq:l-f})
\begin{align*} \Lambda(f) + C = \Lambda(f + C\Id{\cX})&\geq \int_{\cX} (f(x) + C) \mu(dx) = \int_{\cX} f(x) \mu(dx)+ C,\\ 
	-\Lambda(f) + C = \Lambda(-f + C\Id{\cX})&\geq \int_{\cX} (-f(x) + C) \mu(dx) = -\int_{\cX}  f(x) \mu(dx)+ C,
\end{align*}
or $\Lambda(f) = \int_{\cX} f(x) \mu(dx), \quad f\in CB(\cX,\tau)$.

We can also prove the uniqueness. Suppose $\mu$ is tight, $\nu$ is a Borel probability measure and 
\[ \int_{\cX} f(x) \mu(dx) =  \int_{\cX} f(x) \nu(dx),\quad f\in CB(\cX,\tau).\]

\begin{lemma}\label{compid}
	If $K$ is a $\tau$-compact subset of a submetric space $\big( \cX, \tau)$ then there exists a sequence $\{f_n\}$ of $\tau$-continuous functions such that $1 \geq f_n(x) \searrow \Id{K}(x)$. 
\end{lemma}
\begin{proof}[Proof of Lemma~{\upshape\ref{compid}}] Set
	\begin{equation}\label{funcrho} f_n(x) = 1 - \big(n d(x,K)\big)\wedge 1, n = 1,2 \ldots.\end{equation}
\end{proof} 

Choose $K \in \cK_{\tau}$ and let $\{f_n\}$ be the sequence described in Lemma \ref{compid}.  Then by the Lebesgue theorem
\[ \mu(K) = \lim_{n\to\infty} \int_{\cX} f_n(x) \mu(dx) = \lim_{n\to\infty} \int_{\cX} f_n(x) \nu(dx) = \nu(K).\]
In particular, $\nu$ is tight and both measures coincide ({\em at least}) on Borel sets  $\zB_{\cX}$.
\end{proof}

\begin{example}\label{ExSmir}
The Smirnov space is the most popular example of a Hausdorff space which is not a regular topological space (see \cite[Example 1, p. 107]{Smir51}, \cite[Example 1.5.6, p. 39]{Eng89}). Let $\cX = [0,1]$ and let $\cZ = \{ 1/n\,;\, n=1,2,\ldots \}$. Denote $U_i(x) = (x - (1/i), x + (1/i))\cap [0, 1]$, $x\in [0,1], i \in \GN$. We define  $\tau$ by a system of neighborhoods:
\[ B(x)  = \begin{cases} \{U_i(x)\}_{i\in \GN} & \text{ if $x \in (0,1]$}; \\
	 \{U_i(x) \setminus \cZ\}_{i\in \GN} & \text{ if $x = 0$}.
	 \end{cases}
	 \]
Then $\cZ$ is a {\em closed} set that cannot be separated from $0$: if $U,V \in \cO_{\tau}$, $0 \in U$, $\cZ \subset V$, then necessarily $U\cap V \neq \emptyset$. 

Notice that this topological space becomes {\em submetric}, when we choose as $d$ the usual distance on $[0,1]$. Moreover, $(\cX,\tau)$ {\em is not compact}. Indeed, if $\{x_n\sr n\in\GN\} \subset \cZ$ and $x_n\to 0+$, then no subsequence of $\{x_n\}$ can be $\tau$-convergent to $0$, as $\cZ$ is closed and $0\not\in\cZ$.

We will show that $CB([0,1],\tau) = C([0,1])$. Let $f \in CB([0,1],\tau)$. Due to the system of neighborhoods that defines $\tau$, $f$ is continuous in the usual sense  at every $x\in (0,1]$. Hence it is enough to check the sequential $d$-continuity of $f$ at $x = 0$. Suppose that $x_n \to 0+$ and that  the set $\{x_n\sr n\in\GN\}\cap \cZ$ is finite.  Take any $\varepsilon > 0$. There exists $i\in \GN$ such that $[0,1/i)\setminus \cZ \subset \{x\sr |f(x) - f(0)| < \varepsilon\}$. It follows that  $|f(x_n) - f(0)| < \varepsilon$ for $n$ large enough. So suppose that along some subsequence $\{n_k\}$ we have $x_{n_k} \in \cZ$, $k\in\GN$. For each $k\in\GN$, replace $x_{n_k}$ with $x_k' \not\in \cZ$ such that $d(x_{n_k},x_k') \to 0$ and $f(x_{n_k}) - f(x_k') \to 0$. For the new sequence $\{x_k'\}$ we have $\{x_k'\sr k\in\GN\}\cap \cZ = \emptyset$, hence by the earlier step $f(0) = \lim_{k\to\infty} f(x_k') = \lim_{k\to\infty} f(x_{n_k})$.

Let $\Lambda$ be a positive linear functional on $CB([0,1],\tau)$. We have just identified $CB(\cX,\tau)$ with $C([0,1])$, hence $\Lambda$ is determined by the classical Riesz representation.
\begin{equation}\label{eq:classic} \Lambda(f) = \int_{[0,1]} f(x)\,\mu(dx), \quad f \in C([0,1]),
\end{equation}
for some Borel measure $\mu$  on $\zB_{[0,1]}$.
Notice that this conclusion is in perfect agreement with our Theorem \ref{main}. Since there are ``less" compact subsets in $([0,1],\tau)$ than in $([0,1],\tau_d)$, there are also ``less" functions in $V_{\infty}(\tau)$ than in  $V_{\infty}(\tau_d)$. It follows that any strictly $\tau$-continuous linear functional is also strictly $\tau_d$-continuous  and therefore is given by (\ref{eq:classic}).

On the other hand, a probability measure $\mu$ on $([0,1],\zB_{[0,1]})$ (in fact, a $\tau_d$-tight measure) defines a strictly $\tau$-continuous linear functional. For example,  take numbers $a_n\searrow 0$ such that $\mu\big((\{0\} \cup [a_n, 1])^c\big) < 1/4^{n+1}$. Define $K_0 = \{0\}$ and 
$K_n = \{0\} \cup [a_n, 1] \in \cK_{\tau}$, $n = 1,2, \ldots$ and consider
\[ v(x) = \begin{cases} 1,& \text{ if $x=0$ or $(1/4) \leq x \leq 1$};\\
	2^{-n}, & \text{ if $ K_n \setminus K_{n-1}$},\quad n=1,2,\ldots.
\end{cases}\]
Clearly, $v$ ``vanishes at infinity" (i.e. $v\in V_{\infty}(\tau)$). Repeating step by step the computation behind formula (\ref{nes}), we obtain the inequality
\[ |\int_{[0,1]} f(x)\,\mu(dx)| \leq \|f v\|_{\infty}.\] 
\end{example}
 \begin{remark}\label{rem:comm}
 	The preceding example explicitly shows that certain irregularities of the original topology cannot be detected by continuous functions and are thus irrelevant in the context of the Riesz representation theorem. Furthermore, even if $CB(\cX,\tau) = CB(\cX,\tau')$, the corresponding families of compact sets $\cK_{\tau}$ and $\cK_{\tau'}$ may still differ.
 \end{remark}
 \begin{remark}\label{rem:Kiiski}
 For any submetric space $(\cX,\tau)$, let $\tau_c$ denote the canonical submetric topology—that is, the coarsest topology rendering all elements of $CB(\cX,\tau)$ continuous. In particular, $\tau_c \subset \tau$. The topology $\tau_c$ is submetric since any metric $d$ that is separately $\tau$-continuous is also separately $\tau_c$-continuous. Fixing a metric $d$ compatible with $\tau$, the space $CB(\cX,\tau_d)$ separates points in $\cX$, implying that $CB(\cX,\tau)$ separates points as well. Consequently, $(\cX,\tau_c)$ is a Tychonoff space, enabling the application of Sentilles' theory. 
 
 Suppose that $\Lambda$ satisfies the conditions of Theorem~\ref{main} with $v \in V_{\infty}(\tau)$. Since every $\tau$-compact set is $\tau_c$-compact, it follows that $V_{\infty}(\tau) \subset V_{\infty}(\tau_c)$. By \cite[Theorem 3, p. 141]{Jar81}, there exists a $\tau_c$-tight measure $\mu$ on the Borel $\sigma$-algebra $\zB_{\tau_c}$ (notice that {\em a priori} we have only  $ \zB_{\tau_c} \subset\zB_{\tau}$). A key ingredient in our argument is the Urysohn-type lemma for compact subsets given in Lemma~\ref{Separ}. Define $K_n = \{ x \in \cX \mid v(x) \geq 1/n \} \in \cK_{\tau} \subset \cK_{\tau_c}$.  For any $\tau_c$-compact subset $L \subset K_n^c$, let $h_n$ be the $\tau_c$-continuous function provided by Lemma~\ref{Separ} for the pair $K_n$ and $L$. Analogously to (\ref{eq:conc1}), we obtain
 \[\mu(L) \leq \int_{\cX} h_n(x)\,\mu(dx) = \Lambda(h_n) \leq \sup_{x \in \cX} \{ h_n(x) v(x) \} \leq 1/n.\]
 Because $\mu$ is a $\tau_c$-tight measure, it follows that $\mu(K_n^c) \leq 1/n$ for each $n \geq 1$. Therefore, $\mu$ is also $\tau$-tight and admits a unique extension to a $\tau$-tight measure on $\zB_{\tau}$.
 
 \end{remark}
 \begin{remark}
 We developed the theory for positive linear functionals. However, the set $CB(\cX,\tau)$ is a vector lattice of bounded functions, and it is therefore standard to extend Theorem~\ref{main} to general real (or complex) linear functionals that are continuous with respect to the strict $\tau$-topology (see, e.g., \cite[Theorem 7.8.3, p. 102]{Bog07II}).
 
 \end{remark}

\section{Uniform tightness and the Skorokhod parametrization}\label{SecUT}

Let us return  for a while to equations (\ref{eq:conc1}) and (\ref{eq:conc2}). They describe the concentration of the measure $\mu$ in terms of the function $v$: if $K_{\varepsilon} = \{x\,:\, v(x) \geq \varepsilon\}\in \cK_{\tau}$ then $\mu(K_{\varepsilon}^c) \leq \varepsilon$. Recall that a family $\{\mu_i\}_{i\in\GI}$ of probability measures on a topological space $(\cX, \tau)$ is {\em uniformly $\tau$-tight} if for every $\varepsilon > 0$ there exists  a compact set $K_{\varepsilon} \in \cK_{\tau}$ such that
\[ \mu_i (K_{\varepsilon}^c) \leq \varepsilon,\quad i \in \GI.\]
We have an immediate consequence of Theorem \ref{main} and the preceding remark.

\begin{theorem} \label{equicont}
	Let $(\cX,\tau)$ be a submetric space and let $CB(\cX,\tau)$ be the space of bounded continuous functions on  $(\cX,\tau,\rho)$. Consider a family $\{\Lambda_i\}_{i\in \GI}$ of positive linear functionals on $CB(\cX,\tau)$, with $\Lambda_i(\Id{\cX}) = 1$, $i\in \GI$. 
	
	Then  there is a $v\in V_{\infty}(\tau)$ such that 
	\[ |\Lambda_i(f)| \leq \|fv\|_{\infty},\quad f\in CB(\cX,\tau),\quad i \in \GI,\]
	(i.e. the family $\{\Lambda_i\}_{i\in \GI}$ is equicontinuous in the strict topology) if, and only if, the corresponding  probability measures $\{\mu_i\}_{i\in\GI}$ are uniformly $\tau$-tight.
\end{theorem}

\begin{example}\label{ExHilb} 
	Let us consider a separable Hilbert space $\GH$ with the weak topology $\tau_w$, i.e. the coarsest topology with respect to which all scalar products $\langle \cdot,y\rangle$, $y\in \GH$, are continuous. By the very definition, $(\GH,\tau_w)$ is a linear topological space, hence Tychonoff. We know also (see (\ref{eq:4})) that it is a submetric space, so both the standard methods based on Tychonoff/Stone/\v{C}ech formalism and our simplification work. 
	
	Take any positive linear functional $\Lambda_i$ on $CB(\GH,\tau_w)$, which satisfies $\Lambda_i(\Id{\GH}) = 1$ and 
	\[	|\Lambda_i(f)| \leq \|fv\|_{\infty},\quad f\in CB(\GH,\tau_w),\]
	for some  $v \in V_{\infty}(\tau_w)$. Then there exists a $\tau_w$-tight probability measure $\mu_i$ on $\zB_{\tau_w}$ such that on $CB(\GH,\tau_w)$ we have $\Lambda_i(f) = \int_{\GH} f(x)\,\mu(dx)$. But it is well-known (see, e.g. \cite[Proposition 2.1]{ItNi68}) that on $\GH$ we have $\zB_{\tau_w} =  \zB_{\|\cdot\|}$ and therefore $\mu_i$ is also tight in the sense of $(\GH,\|\cdot\|)$. Hence {\em individually} $\Lambda_i$ satisfies a much stronger inequality
	\[  |\Lambda_i(f)| \leq \|f u_i\|_{\infty},\quad f\in CB(\GH,\|\cdot\|) \supset CB(\GH,\tau_w),\]
	with $u_i \in V_{\infty}(\|\cdot\|) \subset V_{\infty}(\tau_w)$.
	However, if we assume 
	\[	|\Lambda_i(f)| \leq \|fv\|_{\infty},\quad f\in CB(\GH,\tau_w),\quad i \in \GI,\]
	then in general we cannot replace $v \in V_{\infty}(\tau_w)$ with some $u \in V_{\infty}(\|\cdot\|)$.
	
	Let us notice that checking uniform $\tau_w$-tightness of some family $\{\mu_i\}$ is particularly simple:
	\[ \lim_{n\to\infty} \sup_i\ \mu_i(\{x\sr \|x\| > n\} ) = 0.\]
	We are going to show that this weak assumption is informative and there is a tool making it quite operational.
\end{example}

\begin{remark}\label{rem:unift} 
	
The notion of uniform tightness was introduced by Prohorov \cite{Pro53, Pro56}, who called it "condition ($\kappa$)" and proved that, in Polish spaces, it implies the relative compactness of a set of probability measures in the weak topology (the direct Prohorov theorem; in fact, this holds in arbitrary metric spaces). The converse Prohorov theorem states that, in Polish spaces, weak relative compactness of a set of probability measures implies its uniform tightness. 

The applicability of uniform tightness stimulated the extensive development of functional limit theorems, as presented in classic books by Billingsley \cite{Bill68} and Parthasarathy \cite{Part67}, as well as in more recent monographs \cite{Whitt02, JaSh03}.

The direct part of Prohorov's theory also holds for measures on Tychonoff spaces with metrizable compacts, as summarized in the review article by Smolyanov and Fomin \cite{SmFo76}. Note that the submetric spaces discussed in Example \ref{ExSmir} and in \cite{Jak97NS} fall outside this framework. To align with the general scheme of limit theorems, we develop a suitable framework for submetric spaces.
\end{remark}

We will need  the Fernique-Skorokhod parametrization of probability measures on a Polish space $(\cY,\rho)$. As it plays an important role in our considerations, we restate here this result.  Let us denote by $\cP(\cY)$ the space of probability measures on $(\cY,\zB_{\cY})$, by 
$\cL^0([0,1],\cY)$ the space of Borel measurable maps from $[0,1]$ to $\cY$, by $\ell$ the Lebesgue measure on $[0,1]$ and by $\Rightarrow$ the weak convergence of probability measures.
\begin{theorem}[Fernique,\cite{Fern88}]\label{th:Fern}
	Let $(\cY,\rho)$ be a Polish space. There exists a map 
	$ \cP(\cY) \ni \nu \mapsto Y_{\nu} \in \cL^0([0,1],\cY)$
	such that 
	\begin{description}
		\item{\bf (i)} $ Y_{\nu} \sim \nu$ for each $\nu \in \cP(\cX)$.
		\item{\bf (ii)} $\nu_n \Rightarrow \nu_{\infty}$ if, and only if, $\rho\big(Y_{\nu_n}(\omega), Y_{\nu_{\infty}}(\omega)\big) \to 0$, for $\ell$-almost all $\omega \in [0,1]$.
	\end{description}
\end{theorem}

\begin{remark}
Before Fernique, Blackwell and Dubins \cite{BlDu83} constructed a Skorokhod parametrization defined on $[0,1]\times[1,2]$. However, their proof was rather sketchy (in contrast to Fernique's rigorous proof) and, as indicated in \cite{Jak25}, contains a serious gap. Note that Bogachev and Kolesnikov \cite{BoKo01} obtained a very general parametrization theorem in metric spaces using advanced tools from topology and functional analysis. While elegant, their proofs are not easily accessible (see \cite{Bog18}, Sections 2.6 and 5.4).
\end{remark}

The main result of this section is a Skorokhod parametrization of a uniformly tight family of probability measures on a submetric space. We need some terminology. 

\begin{definition}
A sequence $\{X_n\sr n=1,2,\ldots, \infty\}$ of random elements with values in a submetric space $(\cX,\tau)$ is said to be convergent almost surely {\em in compacts} if 
\[\bP\big( \omega\sr X_n(\omega) \intau X_{\infty}(\omega)\big) = 1\] and for each $\varepsilon > 0$ there is a $\tau$-compact set $K_{\varepsilon}$ such that
\[  	\bP\big(\{\omega \in [0,1]\sr X_{n}(\omega) \in K_{\varepsilon}, n = 1,2,\ldots,\}\big) > 1 - \varepsilon.\]   
\end{definition}
\begin{remark}
In other words, $X_n \to X_{\infty}$ almost surely in compacts if $X_n \conv_{a.s.} X_{\infty}$ and $\{X_n\}$ is {\em strongly} tight. This is not a very exotic mode of convergence. For example, an almost surely convergent sequence of elements  in a Polish space is strongly tight (see \cite[Theorem 3.1]{KruZie96}); hence, it is  almost surely convergent in compacts.	
\end{remark}
\begin{definition}  We say that a family $\{\mu_i\}_{i\in \GI}$ of tight measures on a submetric space $(\cX,\tau)$ admits a Skorokhod parametrization if there exists a family $\{Y_i\}_{i\in\GI}$ of random elements, defined on $([0,1],\zB_{[0,1]}, \ell)$ and with values in $(\cX,\tau)$, such that 
\begin{description}
		\item{\bf (i)} $Y_i \sim \mu_i$,\quad $i\in\GI$;
		\item{\bf (ii)} In each subsequence $\{Y_n\}_{n\in\GN}\subset \{Y_i\}_{i\in\GI}$ one can find a further subsequence $\{Y_{n_k}\}_{k\in\GN}$ and a random element $Y_{\infty}$ on $([0,1],\zB_{[0,1]}, \ell)$  such that $Y_{n_k} \conv Y_{\infty}$ almost surely in compacts.
	\end{description}
\end{definition}

\begin{theorem}\label{th:SkPar}
	If $\{\mu_i\}_{i\in\GI}$ is a uniformly $\tau$-tight family of probability measures on a submetric space $(\cX,\tau)$,
	then it admits a Skorokhod parametrization.
\end{theorem}

\begin{lemma}\label{lem:sigmacomp}
	Let $K_1, K_2, \ldots $, be compact subsets of a submetric space $(\cX,\tau)$. There exists a continuous mapping $\widetilde{f} : \cX \to \GR^{\infty}$ which is {\em a measurable isomorphism} between   $\cup_{m=1}^{\infty} K_m \subset \cX$ and $\cup_{m=1}^{\infty} \widetilde{f}(K_m) \subset \GR^{\infty}$ and {\em a homeomorphism}  between every $\tau$-compact $K\subset \cup_{m=1}^{\infty} K_m$ and  $\widetilde{f}(K) \subset \GR^{\infty}$.
\end{lemma}
\begin{proof}[Proof of the lemma]
	Let $d$ be a $\tau$-compatible metric on $\cX$. Fix $m\in\GN$. Since $K_m$ is a metric compact, it is a separable metric space and we can find a countable dense subset $\{ z_{m,j}\,;\, j\in\GN\} \subset K_m$.
	Notice that if $x,y \in K_m$, then  $\rho(x,z_{m,j}) = \rho(y,z_{m,j}),\ j\in \GN$, implies $x = y$.
	
	Let $\{z_1, z_2, \ldots \}$ be any enumeration of  the countable set $\cup_{m=1}^{\infty} \{ z_{m,j}\,;\, j\in\GN\}$. Define
	\[ \widetilde{f}(x) = \big(d(x,z_1), d(x,z_2),\ldots \big) \in \GR^{\infty}.\]
	Clearly, $\widetilde{f}$ is continuous. It is also a one-to-one map when restricted to $\cup_{m=1}^{\infty} K_m$, hence is a homeomorphism when restricted to any $\tau$-compact $K \subset  \cup_{m=1}^{\infty} K_m$ (see, e.g., \cite[Theorem 3.1.13, p. 125]{Eng89}). It remains to establish the measurable isomorphism property. Let $F \subset \cX$ be $\tau$-closed. Then 
	\[\widetilde{f}\big(F \cap \cup_{m=1}^{\infty} K_m\big) = \cup_{m=1}^{\infty} \widetilde{f}\big(F\cap K_m\big) \in \cB_{\GR^{\infty}},\]
	for $\widetilde{f}\big(F\cap K_m\big)$ are compacts
	in $\GR^{\infty}$.
\end{proof}

\begin{proof}[Proof of Theorem \ref{th:SkPar}]
For the sake of clarity of notation, consider  a family of random variables $X_i \sim \mu_i$, $i\in \GI$.  
Let $\{K_m\}_{m\in\mathbb{N}}$ be $\tau$-compact sets such that $K_m \subset K_{m+1}$ for all $m\in\mathbb{N}$, and
\[ \mathbb{P}\big(X_i \in K_m\big) > 1 - 1/m, \qquad i\in I. \]

Let $\widetilde{f}$ be the mapping of $\mathcal{X}$ into $\mathbb{R}^{\infty}$ associated with the sequence $K_1, K_2, \ldots$ by Lemma \ref{lem:sigmacomp}. Define $\Phi : \mathbb{R}^{\infty} \to \mathbb{N} \cup \{+\infty\}$ by
\[ \Phi(\bar{y}) = \inf \{ m \ge 1 \sr \bar{y} \in \widetilde{f}(K_m) \}, \]
where, as usual, $\inf \emptyset = +\infty$. Note that $\Phi$ is finite on $\bigcup_{m=1}^{\infty} \widetilde{f}(K_m)$.

Let us introduce the auxiliary random elements
\[ \widetilde{Z}_i = \big( \widetilde{f}(X_i), \Phi(\widetilde{f}(X_i))\big) \]
almost surely taking values in $\mathbb{R}^{\infty} \times \mathbb{N}$. Set
\[ \widetilde{K}_m = \widetilde{f}(K_m) \times \{1, 2, \dots, m\}. \]
Then 
\begin{align*}
	\mathbb{P}\big(\widetilde{Z}_i \in \widetilde{K}_m\big) &=
	\mathbb{P}\big(\widetilde{f}(X_i) \in \widetilde{f}(K_m),\, \Phi(\widetilde{f}(X_i)) \leq m\big) \\
	&= \mathbb{P}\big( \widetilde{f}(X_i) \in \widetilde{f}(K_m)\big) = \mathbb{P}\big( X_i \in K_m\big) > 1 - 1/m,
\end{align*}
and hence the family $\{\widetilde{Z}_i\}$ is uniformly tight on the Polish space $\mathbb{R}^{\infty} \times \mathbb{N}$.

Let $\widetilde{Y}_i \sim \widetilde{Z}_i$, $i\in I$, be the elements of the Fernique--Skorokhod parametrization of probability measures from $\mathcal{P}\big( \mathbb{R}^{\infty} \times \mathbb{N}\big)$ defined on $([0,1],\mathcal{B}_{[0,1]},\ell)$. Let $\Pi_{\mathbb{R}^{\infty}}$ and $\Pi_{\GN}$ be the projections onto the factors of the product space $\mathbb{R}^{\infty}\times \mathbb{N}$. 
Let us observe that 
\begin{equation}\label{eq:eq} 
	\bP\big( \big\{\omega\sr \widetilde{Y}_i(\omega) = \big(\Pi_{{\GR}^{\infty}}(\widetilde{Y}_i(\omega)), \Phi\big(\Pi_{{\GR}^{\infty}}(\widetilde{Y}_i(\omega))\big)\big)\big\}\big) = 1,\qquad  i\in\GI.
	\end{equation}
Indeed, if $(U,\Psi(U)) \sim (V,W)$, then $(\Psi(U),\Psi(U)) \sim (\Psi(V),W)$, hence  $\bP\big( \Psi(V) = W\big) = 1$, provided the mapping $\Psi$ takes values in a separable metric space.
 
We claim that
\[ Y_i = \widetilde{f}^{-1} \big(\Pi_{\mathbb{R}^{\infty}} (\widetilde{Y}_i)\big), \qquad i\in I, \]
is the desired Skorokhod parametrization of $\{\mu_i\}_{i\in I}$.

First, $Y_i$ is  well-defined ($\ell$-almost surely) on $[0,1]$, as 
\[	\bP\big(\widetilde{Y}_i \in \bigcup_{i=1}^{\infty} \widetilde{K}_m\big)	= \bP\big(\widetilde{Z}_i \in \bigcup_{i=1}^{\infty} \widetilde{K}_m\big) = 1,\] and 
\[ \Pi_{\GR^{\infty}}\Big( \bigcup_{i=1}^{\infty} \widetilde{K}_m\Big) = 
	 \bigcup_{i=1}^{\infty}  \Pi_{\GR^{\infty}}\Big(\widetilde{K}_m\Big) =  \bigcup_{i=1}^{\infty} \widetilde{f}(K_m).\]
Moreover, for any $\tau$-compact $K \subset \bigcup_{k=1}^{\infty}  K_m$ we have 
\begin{align*}
	\bP \big( Y_i \in K\big) &= \bP \big( \widetilde{f}(Y_i) \in \widetilde{f}(K)\big) = \bP \big( \Pi_{\GR^{\infty}}(\widetilde{Y}_i) \in \widetilde{f}(K)\big)= \bP \big(\widetilde{Y}_i \in \widetilde{f}(K)\times \GN)\\
	&= \bP \big(\widetilde{Z}_i \in \widetilde{f}(K)\times \GN) =\bP \big( \widetilde{f}(X_i) \in \widetilde{f}(K)\big) = 	\bP \big( X_i \in K\big).
	\end{align*}
	It follows that the family $\{Y_i\sr  i\in\GI\}$ is a coupling of $\{X_i\sr i\in\GI\}$. In order to complete the proof, take any sequence
	$\{Y_n\} \subset  \{Y_i\sr  i\in\GI\}$. The sequence $\{\widetilde{Y}_n\}$ of random elements with values in $\GR^{\infty} \times \mathbb{N}$ is uniformly	tight.
	By the Direct Prohorov Theorem \cite{Pro56} there exists a subsequence $\{\widetilde{Y}_{n_k}\}$ which is weakly convergent to some $\widetilde{Y}_{\infty}$. By the property (ii) of the Fernique-Skorokhod parametrization (Theorem \ref{th:Fern}), and relation (\ref{eq:eq}) we obtain
	\begin{align*}
		 1 &= \bP\big(\omega\sr \widetilde{Y}_{n_k}(\omega) \to \widetilde{Y}_{\infty}(\omega)\text{ in $\GR^{\infty} \times \mathbb{N}$}\big)\\
		  &= \bP\big(\omega\sr \widetilde{f}({Y}_{n_k}(\omega)) \to \Pi_{\GR^{\infty}}(\widetilde{Y}_{\infty}(\omega)), \Phi(\widetilde{f}({Y}_{n_k}(\omega)))\to \Pi_{\GN}(\widetilde{Y}_{\infty}(\omega))  \big).
	\end{align*}
	Set $\Omega_0 = \{ \omega\sr \widetilde{f}({Y}_{n_k}(\omega)) \to \Pi_{\GR^{\infty}}(\widetilde{Y}_{\infty}(\omega)\}$ and let $\omega_0 \in \Omega_0$. 
	Note that $\Phi(\widetilde{f}({Y}_{n_k}(\omega_0)))\to \Pi_{\GN}(\widetilde{Y}_{\infty}(\omega_0)) < +\infty$ implies that
	$\widetilde{f}({Y}_{n_k}(\omega_0)) \in \widetilde{f}(K_{M(\omega_0)})$, where 
	\[M(\omega_0) = \max\{\Phi(\widetilde{f}({Y}_{n_k}(\omega_0)))\sr k\in\GN\} < +\infty.\]
	 It follows that $\Pi_{\GR^{\infty}}(\widetilde{Y}_{\infty}(\omega_0)) \in \bigcup_{m=1}^{\infty} \widetilde{f}(K_m)$ and we can define $Y_{\infty}(\omega_0) = \widetilde{f}^{-1}\big(\Pi_{\GR^{\infty}}(\widetilde{Y}_{\infty}(\omega_0))\big)$. Moreover, $Y_{n_k}(\omega_0) \intau Y_{\infty}(\omega_0)$, since the convergence $\widetilde{f}(Y_{n_k}(\omega_0)) \to \widetilde{f}(Y_{\infty}(\omega_0))$ holds inside the compact set 
	$\widetilde{f}(K_{M(\omega_0)})$. 
	
	We have proved that $Y_{n_k} \intau Y_{\infty}$ almost surely. The strong tightness also holds, since $\Omega_0 \subset \bigcup_{r=1}^{\infty} \{\omega\sr M(\omega) \leq r\}$.
\end{proof}
As a corollary we obtain the strong version of the direct Prohorov theorem, extending \cite{Jak97S} to general submetric spaces.
\begin{corollary}\label{cor:DPT}
	If $\{\mu_i\}_{i\in\GI}$ is a uniformly $\tau$-tight family of probability measures on a submetric space $(\cX,\cK_{\tau})$, then in each sequence $\{\mu_n\} \subset \{\mu_i\}_{i\in\GI}$ one can find a further subsequence $\{\mu_{n_k}\}$ which admits the almost sure Skorokhod representation, i.e., there exists a sequence $\{Y_k\}$, $k=1,2,\ldots, \infty$ of $\cX$-valued random elements, defined on a common probability space $(\Omega,\cF,\bP)$, and such that
	\begin{description}
		\item{\bf (i)} $Y_k \sim \mu_{n_k}$,\quad $k\in\GN$;
		\item{\bf (ii)} $Y_k \intau Y_{\infty}$ $\bP$-almost surely.
	\end{description}
\end{corollary}
\begin{remark}\label{SCM}
	The role of weak topologies in deterministic existence problems (e.g., for solutions to PDEs) is well-established. Interestingly, methods based on Corollary~\ref{cor:DPT} (originally on \cite{Jak97S}) have become a standard tool in the stochastic analog of this theory. The stochastic compactness method, systematically presented in \cite{BFH18}, provides a general framework frequently applied to construct solutions to stochastic partial differential equations (SPDEs) and to study other problems involving complex stochastic processes, such as the existence of invariant measures.
\end{remark}

\section{The topology proper for uniform tightness in submetric spaces}\label{SecLaw}

In metric spaces, the direct Prohorov theorem asserts that uniform tightness of a family of probability measures implies relative compactness in the sequential topology of weak convergence. Corollary \ref{cor:DPT} shows that the same holds in submetric spaces. Indeed, if $(Y_k)$ is the Skorokhod representation for a selected subsequence $(\mu_{n_k})$, then for any $f \in CB(\cX,\tau)$, the Lebesgue dominated convergence theorem yields
\[ 
\lim_{k\to\infty} \int_{\cX} f(x)\,\mu_{n_k}(dx) = \lim_{k\to\infty} \bE[f(Y_k)] = \bE[f(Y_{\infty})] = \int_{\cX} f(x)\,\mu_{\infty}(dx).
\]
The converse statement fails in a manner suggesting that, in submetric spaces, weak convergence of probability measures is not the appropriate concept to consider in the context of uniform tightness.

Fernique \cite[Exemple 1.6.4, p.~24]{Fern67} constructed a sequence of probability measures $(\mu_n)$ on a separable Hilbert space $\GH$ equipped with the weak topology $\tau_w$ such that $\mu_n \Rightarrow_{\tau_w} \delta_0$, yet for every $R > 0$,
\[ 
\lim_{n\to\infty} \mu_n(\{x \sr \|x\| > R\}) = 1.
\]
In particular, no subsequence $(\mu_{n_k})$ of $(\mu_n)$ is uniformly $\tau_w$-tight.

In fact, the conclusion of Corollary \ref{cor:DPT} is much stronger than weak convergence of the underlying distributions. To see this, assume that a sequence $(\mu_n)$ of probability measures on $(\GH,\tau_w)$ admits a Skorokhod representation $Y_n$. 
For almost all $\omega$, we have $Y_n(\omega) \xrightarrow{\tau_w} Y_{\infty}(\omega)$, and therefore $\sup_n \|Y_n(\omega)\| < +\infty$. This implies the strong $\tau_w$-tightness of $\{Y_n\}$, and hence uniform $\tau_w$-tightness of $\{\mu_n\}$. Note that this conclusion is directly stated in Theorem \ref{th:SkPar} (relative compactness with respect to almost sure convergence in compacts).

Let $(\cX,\tau)$ be a submetric space. We will construct a sequential topology on the set $\cP(\cX,\tau)$ of $\tau$-tight probability measures that, in submetric spaces, is better suited to Theorem \ref{th:SkPar} than the weak topology.
For details on the general scheme for constructing sequential topologies from so-called $\cL$-convergences, we refer to Section 6 of \cite{Jak18}.

\begin{definition}  
	Let $(\cX,\tau)$ be a submetric space. We say that a sequence $\{\mu_n\} \subset \cP(\cX,\tau)$ converges to $\mu_{\infty}$ {\em via coupling} if, on some probability space $(\Omega,\cF,\bP)$, there exists a representation $\{Y_n \sim \mu_n \sr n = 1, 2, \ldots, \infty\}$ that converges $\bP$-almost surely in compacts. In this case, we write 
	$\mu_n \Rrightarrow \mu_{\infty}$.
\end{definition}

It is easy to see that $\Rrightarrow$ is an $\cL$-convergence in the sense of Fréchet.
Consequently, $\Rrightarrow$ defines a sequential topology $\tau(\Rrightarrow)$ on $\cP(\cX,\tau)$. It is straightforward to show that in metric spaces, the topology $\tau(\Rrightarrow)$ coincides with the standard sequential topology of weak convergence of probability measures on $\cP(\cX,\tau)$. Moreover, the topology $\tau(\Rrightarrow)$ on $\cP(\cX,\tau)$ is submetric. To see this, take a compatible metric $d$ and consider, for example, Prohorov's metric:
\[\rho(\mu,\nu) =\inf\bigl\{\varepsilon > 0 \sr \mu(A) \le \nu(A^{\varepsilon}) + \varepsilon \text{ and } \nu(A) \le \mu(A^{\varepsilon}) + \varepsilon \text{ for all } d\text{-closed sets } A \bigr\},\]
where $A^{\varepsilon} = \{x \in \cX \sr d(x,A) < \varepsilon\}$.

We conclude this section with several comments and a representative theorem illustrating the efficacy of the developed theory.

\begin{remark}\label{rem:rec}
	The topology $\tau(\Rrightarrow)$ defines in a natural way a new convergence $\stackrel*{\Rrightarrow}$, termed ``a posteriori'' after Urysohn, to distinguish it from the ``a priori'' convergence $\Rrightarrow$.
	
	By the Kantorovich--Vulich--Pinsker--Kisyński recipe (see \cite[Theorem 6.3]{Jak18}), $\mu_n \stackrel*{\Rrightarrow} \mu_0$ if and only if every subsequence $(\mu_{n_k})_{k\in\mathbb{N}}$ contains a further subsequence $(\mu_{n_{k_l}})_{l\in\mathbb{N}}$ such that $\mu_{n_{k_l}} \Rrightarrow \mu_0$.
\end{remark}

\begin{remark}
Being sequential, the topology $\tau(\Rrightarrow)$ is the finest among all topologies preserving the family of $\tau(\Rrightarrow)$-compact sets. 
See \cite[Theorem 3.11]{Jak26}.
\end{remark}

\begin{remark}
	The topology $\tau(\Rrightarrow)$ is the natural topology for convergence in law of random elements on submetric spaces. This perspective is strongly supported by Theorems 7.1 and 7.3 in \cite{Jak23}, to which we refer the reader for further context and details regarding $\tau(\Rrightarrow)$.
\end{remark}

\begin{remark}
	In light of Remark \ref{rem:rec}, it is natural to change the notion of convergence in law for random elements on submetric spaces. Let $\{X_n \sr n=1,2,\ldots, \infty\}$ be random elements taking values in a submetric space $(\cX,\tau)$ and with $\tau$-tight laws. The sequence $\{X_n\}$ is said to converge in law to $X_0$ if every subsequence $(X_{n_k})_{k\in\mathbb{N}}$ contains a further subsequence $(X_{n_{k_l}})_{l\in\mathbb{N}}$ that admits a Skorokhod representation converging almost surely in compacts.
	
	We propose writing 
	\[ 
	X_n \starindistdisttop{\tau} X_0,
	\]
	whereas 
	\[ 
	X_n \indistdisttop{\tau} X_0 
	\]
	indicates that $\{X_n\}$ itself admits a Skorokhod representation.
\end{remark}

\begin{theorem}
	Let $(\GH,\tau_w)$ be as in Example \ref{ExHilb}. Then a family $\{\mu_i \sr i\in I\}$ of probability measures on $(\GH,\tau_w)$ is uniformly $\tau_w$-tight if and only if it is relatively compact with respect to $\Rrightarrow$.
\end{theorem}

\begin{proof}
	Recall that a family $\{\mu_i \sr i\in I\}$ of probability measures on $(\GH,\tau_w)$ is uniformly $\tau_w$-tight if
	\[ 
	\lim_{R\to\infty} \sup_{i\in I} \mu_i(\{x \sr \|x\| > R\}) = 0.
	\]
	Necessity is established in Theorem \ref{th:SkPar}. To prove sufficiency, suppose that $\{\mu_i \sr i\in I\}$ is relatively compact with respect to $\Rrightarrow$, but there exist $\varepsilon > 0$ and a sequence $\{\mu_N\}$ such that $\mu_N(\{x \sr \|x\| > N\}) \ge \varepsilon$. In particular, no subsequence of $\{\mu_N\}$ can be uniformly $\tau_w$-tight. However, we have already shown that any sequence admitting a Skorokhod representation must be uniformly $\tau_w$-tight, which contradicts the assumed relative compactness.
\end{proof}

\begin{remark}
	In the special case where the submetric structure is induced by a countable family of point-separating continuous functions, the topology $\tau(\Rrightarrow)$ was constructed in \cite{Jak00}.
\end{remark}

\appendix
\section{Halmos' content in Hausdorff spaces}\label{ApHalm}
 In \cite[\S 53]{Halm50} Halmos gives a method of construction of Borel measures on locally compact spaces via  the notion of so-called ``content". In Halmos' reasoning there are parts essentially depending on the local compactness (e.g. the notion of ``regular content"), but most of the arguments use general measure-theoretic tools and  relations between compactness and openness in Hausdorff topological spaces. In this Appendix we provide an excerpt form Halmos' theory in order to construct {\em finite} measures on Hausdorff spaces.

In this Appendix $\big(\cX,\tau)$ is a Hausdorff space with $\tau$-compact subsets $\cK_{\tau}$ and $\tau$-open subsets $\cO_{\tau}$. Elements of $\cK_{\tau}$ are denoted by $K, K_1, K_2, C, C_1, C_2$, etc. Elements of $\cO_{\tau}$ are denoted by $U, U_1,U_2, V, V_1, V_2$, etc.

\begin{definition}[\bf Content]\label{halmcont}
	A $\tau$-content is a function $\lambda: \cK_{\tau} \to \GR^+$ with the following properties:
	\begin{enumerate}
		\item $0 \leq \lambda(K) \leq  \sup_{K\in \cK_{\tau}} \lambda(K) < +\infty$;
		\item If $K_1 \subset K_2$, then $\lambda(K_1) \leq \lambda(K_2)$;
		\item $\lambda(K_1 \cup K_2) \leq \lambda(K_1) + \lambda(K_2)$;
		\item If $K_1 \cap K_2 = \emptyset$, then $\lambda(K_1 \sqcup K_2) = \lambda(K_1) + \lambda(K_2)$.
	\end{enumerate}
\end{definition}

\begin{remark}
	We have always $\lambda(\emptyset) = 0$.
\end{remark}

\begin{definition}[\bf Inner content] \label{inncont}
	The inner $\tau$-content induced by $\lambda$ is a function $\lambda_*: \cO_{\tau}\to \GR^+$ defined by the formula
	\[ \lambda_*(U) = \sup \{\lambda(K)\sr K \subset U\}.\]
\end{definition}

\begin{theorem}{\bf \cite[Theorem A, p. 232]{Halm50}}
	The inner $\tau$-content $\lambda_*$ vanishes on $\emptyset$, is monotone, countably subadditive and countably additive.
\end{theorem}

\begin{definition}[\bf The outer measure]\label{outmeas}
	The outer measure induced by $\lambda$ is defined by
	\[ \mu^*(E) = \inf\{ \lambda_*(U)\sr E \subset U\},\quad E \subset \cX.\]
\end{definition}

\begin{theorem}{\bf \cite[Theorems B and C, pp. 233--234]{Halm50}}
	\label{outhalm} The following statements hold:
\begin{description}
\item{\bf(i)} $\mu^*$ is an outer measure.
\item{\bf(ii)} $\mu^*(U) = \lambda_*(U), \quad U\in \cO_{\tau}$.
\item{\bf(iii)} $\mu^*(K) \geq \lambda(K), \quad K\in \cK_{\tau}$.
\end{description}
\end{theorem}

\begin{theorem}{\bf  \cite[Theorems D and E, pp. 234--235]{Halm50}}\label{halm}
	\begin{description}
\item{}  
\item{\bf(i)}	A subset $E \subset \cX$ is $\mu^*$-measurable if, and only if, for every $U \in \cO_{\tau}$ 
\begin{equation}\label{meas} \mu^*(U) \geq \mu^*(U\cap E) + \mu^*(U \cap E^c).
\end{equation}
\item{\bf(ii)} Compact sets $\cK_{\tau}$ are $\mu^*$-measurable.
	\end{description}
\end{theorem} 

As a consequence of Theorem \ref{halm} and Theorem \ref{outhalm} {(iii)} we obtain
\begin{theorem} \label{apmain}
A  $\tau$-content $\lambda: \cK_{\tau} \to \GR^+$ generates a finite measure $\mu$ on the thin $\sigma$-algebra $\zK_{\tau} = \sigma(\cK_{\tau})$ such that 
\[  \lambda(K) \leq \mu(K),\quad K\in \cK_{\tau}. \] 
\end{theorem} 
\begin{proof} We set $\mu(A) = \mu^*(A)$, $A\in \zK_{\tau}$.
	\end{proof}

\section{Tight measures on submetric spaces}\label{sec:tight}

Measure theory on topological spaces provides many examples showing that intuitions drawn from the metric setting can be misleading. We refer to \cite[Chapter 7]{Bog07II} for a contemporary survey of this theory and its basic definitions (Borel, regular, tight, and Radon measures).

In submetric spaces, however, the situation simplifies and closely mirrors the metric case.

\begin{proposition}\label{prop: unif}
	Let $(\cX,\tau)$ be a submetric space with a family $\cK_{\tau}$ of $\tau$-compact sets and a family $\cO_{\tau}$ of $\tau$-open sets. Let $\mu$ be a probability measure on $\sigma(\cK_{\tau})$  such that
	\begin{equation}\label{eq:tight1} \sup_{K\in \cK_{\tau}} \mu(K) = 1.
	\end{equation}
	Then the following statements hold.
	\begin{description}
		\item{\bf (i)} There exists a Borel probability measure $\widetilde{\mu}$ which is an extension of $\mu$: 
		\[ \widetilde{\mu}(K) = \mu(K),\qquad K\in\cK_{\tau}.\]
		\item{\bf (ii)} The extension $\widetilde{\mu}$ is $\tau$-tight and $\tau$-regular, hence $\tau$-Radon. 
		\item{\bf (iii)} The measure $\widetilde{\mu}$ is the unique $\tau$-regular extension of $\mu$. 
		\item{\bf (iv)} The measure $\widetilde{\mu}$ is the unique $\tau$-Radon extension of $\mu$.
	\end{description}
\end{proposition}

\begin{proof}
	The proof is fairly standard and we provide it in order to highlight the subtleties of definitions from \cite{Bog07II}. 
	
	Let $K_n\in\cK_{\tau}$ be such that $K_n \subset K_{n+1}$ and $\mu(K_n) > 1- (1/n)$. Set $K_{\infty} = \bigcup_{n=1}^{\infty} K_n$. For $B \in \zB_{\tau} = \sigma(\cO_{\tau})$ define 
	\[ \widetilde{\mu}(B) = \mu(B\cap K_{\infty}).\]
	Then $\widetilde{\mu}$ is well-defined and for any $B \subset K_n^c$ we have $\widetilde{\mu}(B) \leq (1/n)$. This means that  $\widetilde{\mu}$ is $\tau$-tight in the sense of \cite[Definition 7.1.4, p. 69]{Bog07II}.
	
	To see that $\widetilde{\mu}$ is $\tau$-regular in the sense of  \cite[Definition 7.1.5, p. 69]{Bog07II}, take $B \in \zB_{\tau}$ and observe that
	\[ \widetilde{\mu}(B) = \lim_{n\to\infty} \mu(K_n \cap B).\]
	For $\varepsilon > 0$, let $n$ bo such that $\mu(K_n \cap B) > \widetilde{\mu}(B) - \varepsilon/2$. $K_n$ is a metric compact and therefore $\mu$ restricted to  $\zB_{K_n}$ is a  regular measure. Let $F \subset K_n \cap B$ be closed in the topology induced by $\tau$ on $K_n$ and such that $\mu(K_n\cap B \setminus F) < \varepsilon/2$. Then $F \subset B$ and $\widetilde{\mu}(B \setminus F) < \varepsilon$. Since $F$ is an intersection of a compact $K_n$ and some $\tau$-closed set and a submetric space is Hausdorff, $F$ is $\tau$-compact, hence also $\tau$-closed. We have got both the regularity of $\widetilde{\mu}$ as well as the fat that it is a Radon measure as defined in \cite[Definition 7.1.1(iii), p. 68]{Bog07II}. 
	
	The uniqueness of the Radon extension is obvious. The uniqueness of the regular extension is implied by the fact that any tight and regular measure is Radon. 
\end{proof}
\begin{remark}
It follows that, in submetric spaces and under condition (\ref{eq:tight1}), there is no ambiguity in referring to  $\mu$ as a tight (or $\tau$-tight) measure.
\end{remark}

\section{Random elements in submetric spaces and their distributions}\label{Ap:RandEl}

In general, establishing the measurability of a mapping $X : (\Omega,\cF,\bP) \to (\cX,\zB_{\tau})$, where $\zB_{\tau} = \sigma(\cO_{\tau})$, can be challenging. 
Checking whether $X^{-1}(K) \in \cF$ for $K \in \cK_{\tau}$ is typically much simpler, since $\tau$-compact sets are metrizable and thus contained in the Borel $\sigma$-algebra $\cB_{\tau_d}$ for some $\tau$-compatible metric $d$. Furthermore, controlling $\bP(X \in K)$ for $K \in \cK_{\tau}$ is sufficient to establish the uniform tightness of a family of distributions.

\begin{definition} 
	Let $\zK_{\tau}$ be the $\sigma$-algebra generated by all compact subsets of $\cX$. 
	\begin{enumerate}
		\item A \emph{random element} in a submetric space $(\cX,\tau)$ is a measurable mapping $X : (\Omega, \cF, \bP) \to (\cX, \zK_{\tau})$.
		\item The \emph{distribution} of a random element $X$ is the probability measure $\bP \circ X^{-1}$ on $\zK_{\tau}$. This distribution is $\tau$-\emph{tight} if, for each $\varepsilon > 0$, there exists a compact set $K_{\varepsilon}$ such that $\bP(X \in K_{\varepsilon}) > 1 - \varepsilon$.
		\item A family of random elements $\{X_i \mid i \in \GI\}$ on $(\cX,\tau)$ is \emph{uniformly $\tau$-tight} (or simply \emph{uniformly tight}) if, for each $\varepsilon > 0$, there exists a compact set $K_{\varepsilon}$ such that 
		\[\bP(X_i \in K_{\varepsilon}) > 1 - \varepsilon \quad \text{for all } i \in \GI.\]
	\end{enumerate}
\end{definition}

\begin{remark}
We denote the set of $\tau$-tight distributions by $\cP(\cX,\tau)$.	
\end{remark}

\begin{theorem} Let $X : (\Omega, \cF, \bP) \to (\cX,\zK_{\tau})$ be a random element with values in a submetric space $(\cX,\tau)$ and with $\tau$-tight distribution. Let $\widetilde{\cF}$ be the completion of $\cF$ with respect to $\bP$ and let $\widetilde{\bP}$ be the natural extension of $\bP$ on $\widetilde{\cF}$. 
	
	Then $X$ is a measurable mapping from $(\Omega, \widetilde{\cF})$ to   
	$(\cX,\zA_{\tau})$, where the $\sigma$-algebra $\zA_{\tau}$ consists of all sets $A$ such that $A\cap K \in \zB_{\tau}$ for each $\tau$-compact set $K$.

In particular, $X$ can be considered as a random element 
$ X : \big(\Omega, \widetilde{\cF},\widetilde{\bP}\big) \to (\cX, \zB_{\tau})$.	
\end{theorem}

\end{document}